\documentclass[11pt]{article}

\usepackage{graphicx,makeidx,multicol,footmisc,dsfont,mathrsfs,amssymb,eucal,amsmath,url,amsthm,fullpage}

\newcommand{\bD}{\mathds{D}}

\newcommand{\dN}{\mathds{N}}
\newcommand{\dR}{\mathds{R}}

\newcommand{\dt}{\, {\rm d}t}

\newcommand{\eoh}{{^1\!/_2\hspace{0.6mm}}}

\newcommand{\Hdiv}{H({\rm div}; \Omega)}

\newcommand{\HNdiv}{H_N({\rm div}; \Omega)}

\newcommand{\Wpm}{H^2(\O)}
\newcommand{\K}{\mathds{K}}

\newcommand{\kds}{\ell}

\newcommand{\mE}{\mathcal{E}}
\newcommand{\mEa}{{\mathcal{E}^j}}
\newcommand{\mEb}{{\mathcal{E}^j_{\partial \Omega}}}
\newcommand{\mEi}{{\mathcal{E}^j_\Omega}}

\newcommand{\mS}{\mathcal{S}}
\newcommand{\mT}{\mathcal{T}}

\newcommand{\nablah}{\nabla_{\! h\,}}

\newcommand{\OT}{\Omega_T}
\newcommand{\pd}{{\mbox{$s$}}}
\newcommand{\pds}{{\mbox{\scriptsize $s$}}}

\newcommand{\pO}{\partial \Omega}

\newcommand{\RT}{{\rm RT}}

\newcommand{\sC}{\mathscr{C}}

\newcommand{\sD}{\mathscr{D}}

\renewcommand{\div}{{\rm div}\,}
\renewcommand{\O}{\Omega}
\renewcommand{\tilde}{\widetilde}
\newtheorem{theorem}{Theorem}
\newtheorem{definition}{Definition}
\newtheorem{numexample}{Numerical Example}

\begin{document}

\title{\bf Stable Crank-Nicolson Discretisation for Incompressible Miscible Displacement Problems\\ of Low Regularity}
\author{Max Jensen\footnote{Mathematical Sciences, University of Durham, England, {\tt m.p.j.jensen@durham.ac.uk}}
, R\"udiger M\"uller\footnote{WIAS, Berlin, Germany, {\tt mueller@wias-berlin.de}}}
\maketitle

\begin{abstract}
In this article we study the numerical approximation of incompressible miscible displacement problems with a linearised Crank-Nicolson time discretisation, combined with a mixed finite element and discontinuous Galerkin method. At the heart of the analysis is the proof of convergence under low regularity requirements. Numerical experiments demonstrate that the proposed method exhibits second-order convergence for smooth and robustness for rough problems.
\end{abstract}

\section{Introduction and Initial Boundary Value Problem}

Mathematical models which describe the miscible displacement of fluids are of particular economical relevance in the recovery of oil in underground reservoirs by fluids which mix with oil. They also play a significant role in CO$_2$ stratification.

This publication extends the analysis of \cite{BartelsJensenMueller08}, which studies the discretisation of miscible displacement under low regularity. Unlike to \cite{BartelsJensenMueller08} which is based on a first-order implicit Euler time-step (leading to a nonlinear system of equations in each time step), here we examine the discretisation in time by a linearised second-order Crank-Nicolson scheme. Crucially, the new, more efficient method inherits stability under low regularity. Like in \cite{BartelsJensenMueller08}, the concentration equation is approximated with a discontinuous Galerkin method, while Darcy's law and the incompressibility condition is formulated as a mixed method. High-order time-stepping for miscible displacement under low regularity has recently also been addressed in \cite{RiviereWalkington09}, however, with a continuous Galerkin discretisation in space and discontinuous Galerkin in time. We refer for an outline of the general literature to \cite{BartelsJensenMueller08,Chen07,Feng02,RiviereWalkington09}.

\begin{definition}[Weak Formulation]
A triple $(u,p,c)$ in
\[
L^\infty(0,T; \HNdiv) \times L^\infty(0,T; L^2_0(\Omega)) \times \bigl( L^2(0,T; H^1(\Omega)) \cap H^1(0,T; \Wpm^*) \bigr)
\]
is called weak solution of the incompressible miscible flow problem if
\begin{enumerate}
\item[{\rm (W1) \hspace*{-6mm}}] \hspace*{6mm} for $t \in (0,T)$, $v \in \HNdiv$ and $q \in L^2_0(\Omega)$
\begin{eqnarray*}
\bigl(\mu(c) \, \K^{-1} u, v \bigr) - \bigl(p, \div v \bigr) &=& \bigl(\rho(c) \, g, v \bigr)\\
\bigl(q, \div u \bigr) &=& \bigl(q^I - q^P, q \bigr).
\end{eqnarray*}
\item[{\rm (W2) \hspace*{-6mm}}] \hspace*{6mm} for all $w \in \sD(0,T; \Wpm)$
\[
\int_0^T - \bigl( \phi \, c, \partial_t w \bigr) + \bigl(\bD(u) \nabla c, \nabla w \bigr) + \bigl(u \cdot \nabla c, w \bigr) + \bigl(q^I c, w \bigr) - \bigl(\hat{c} q^I, w \bigr) {\rm d}t = 0.
\]
\item[{\rm (W3) \hspace*{-6mm}}] \hspace*{6mm} $c(0,\cdot) = c_0$ in $\Wpm^*$.
\end{enumerate}
\end{definition}
For the data qualification we refer to condition (A1)--(A8) in \cite{BartelsJensenMueller08} and for the physical interpretation of the system  to \cite{BartelsJensenMueller08,Chen07,Feng02}. We point out that $\bD$ growths proportionally with $u$:
\[
d_{\circ} (1+|u|) |\xi|^{2} \le \xi^{\sf T} \, \bD(u, x) \, \xi \le d^{\circ} (1+|u|) |\xi|^{2}, \qquad u, \xi \in \dR^{d}, \; x \in \Omega.
\]
Thus $\bD$ is in general unbounded on Lipschitz domains $\Omega$ and in the presence of discontinuous coefficients, which are permitted in this paper.

\section{The Finite Element Method}

We compactly recall the definition of the finite element spaces from \cite{BartelsJensenMueller08}. Let $0 = t_0 <t_1< \ldots<t_M = T$ be a partition of the time interval $[0,T]$. Let $k_j := t_j-t_{j-1}$ and $d_t a^{j} := k_j^{-1} \bigl( a^j - a^{j-1} \bigr)$. We consider meshes $\mT$ of $\O$ with elements $K$ and set $h_K := {\rm diam}(K)$. We denote by $\mS^{\pds}(\mT)$ the space of elementwise polynomial functions of total or partial degree $\pd$. For $w_h \in \mS^{\pds}(\mT)$ the function $\nablah w_h$ is defined through $(\nablah w_h)|_K= \nabla (w_h|_K)$. The sets of interior and boundary faces are $\mE_\O(\mT)$ and $\mE_{\pO}(\mT)$. We set $\mE(\mT) = \mE_\O(\mT) \cup \mE_{\pO}(\mT)$ and assign to each $E \in \mE(\mT)$ its diameter $h_E$. We denote jump and the average operators by $[ \cdot ]$ and $\{ \cdot \}$. The concentration $c$ is discretised at time $j$ on the mesh $\mT^j_c$ or simply by $\mT^j$. The approximation space for the variable $c$ at time step $j$ is denoted by $\mS_c^j$. Often we abbreviate $\mEa := \mE(\mT^j_c)$, $\mEi := \mE_\Omega(\mT^j_c)$, $\mEb := \mE_{\pO}(\mT^j_c)$. We denote the Raviart-Thomas space of order $\ell$ by $\RT^\kds(\mT_u^j)$. The approximation spaces of $u$ and $p$ are $\mS_u^j := \RT^\kds(\mT_u^j) \cap \HNdiv$ and $\mS_p^j := \mS^\kds(\mT_u^j) \cap L^{2}_{0}(\O)$. We frequently use the global mesh size and time step $h^j := \max_{K \in \mT^j_c \cup \mT^j_u} h_K$, $\tilde{h} := \max_{0 \le j \le M} h^j$, $\tilde{k} := \max_{0 \le j \le M} k^j$ as well as to $\mS_u = \prod_{j = 1}^M \mS_u^j, \mS_p = \prod_{j = 1}^M \mS_p^j, \mS_c = \prod_{j = 0}^M \mS_c^j$. In addition we impose conditions (M1)--(M5) of \cite{BartelsJensenMueller08} which are on shape-regularity, boundedness of the polynomial degree, control $\| v_h \|_{L^4} \lesssim \| v_h \|_{H^1}$ and the structure of hanging nodes.

To deal with discontinuous coefficients and the time derivative, we substitute $\bD$ by
\[
\bD_h : L^2(\O)^d \to \mS^{\pds}(\mT_c, \dR^{d \times d}), \; v \mapsto \Pi_\mT \circ \bD(v, \cdot)
 \]
where the $\Pi_\mT$ are projections such that $\| \Pi_\mT \, D \|_K \lesssim \| D \|_K$. Given quantities $a^j$, $a^{j-1}$ and $a^{j-2}$ at times $t_j$, $t_{j-1}$, $t_{j-2}$, we denote $\overline{a}^j = {\textstyle \frac{1}{2}} a^j + {\textstyle \frac{1}{2}} a^{j-1}$ and $\breve{a} = {\textstyle \frac{3}{2}} a^{j-1} - {\textstyle \frac{1}{2}} a^{j-2}$.

The diffusion term of the concentration equation is discretised by the symmetric interior penalty discontinuous Galerkin method: Given $c_h, w_h \in \mS_c^j$, $u_h \in \mS_u^j$, we set
\begin{eqnarray*}
B_d(c_h, w_h; u_h) &:=& \bigr( \bD_h^j(u_h) \nablah c_h, \nablah w_h \bigr) - \bigl( [c_h], \{ \bD_h^j(u_h) \, \nablah w_h \} \bigr)_{\mEi}\hphantom{,}\\&& \quad - \, \bigl( [w_h], \{ \bD_h^j(u_h) \, \nablah c_h \} \bigr)_{\mEi} + \bigr( \sigma^2 [c_h], [w_h] \bigr)_\mEi
\end{eqnarray*}
where $\sigma$ is chosen sufficiently large to ensure coercivity of $B_d$, cf.~\cite{BartelsJensenMueller08}. The convection, injection and production terms are represented by
\begin{eqnarray} \label{defBcq}
B_{cq}( c_h, w_h;&& \!\!\!u_h) := \eoh \Bigl( \bigl(u_h \nablah c_h, w_h \bigr) - \bigl( u_h c_h, \nablah w_h \bigr) + \bigl( (\overline{q}^I + \overline{q}^P) c_h, w_h \bigr)\\
&& + \! \sum_{K \in \mT^j} \bigl( c_h^+, (u_h \cdot n_K)_+ \, [w_h]_K \bigr)_{\partial K \setminus \pO} - \bigl( (u_h \cdot n_K)_- \, [c_h]_K, w_h^+ \bigr)_{\partial K \setminus \pO} \Bigr), \nonumber
\end{eqnarray}
where $(u_h \cdot n)_+ := \max \{ u_h \cdot n, 0\}$ and $(u_h \cdot n)_- := \min \{ u_h \cdot n, 0\}$. We set $B = B_d + B_{cq}$.

{\sc Algorithm $\!(A^{\! dG})$.} {\em Choose $c_h^j \in \mS_c^j$ for $j=0,1$. Given $c_h^j$, find $(u_h^j, p_h^j) \in \mS_u^j \times \mS_p^j$ such that
\begin{equation}\label{scheme_1}
\begin{array}{ccccccc}
&\bigl( \mu(c_h^j) \, \K^{-1} u_h, v_h \bigr)& - &\bigl( p_h , \div v_h \bigr) & \! = \,& \, \bigl( \rho(c_h^j) \, g, v_h\bigr) ,& \\
&\bigl( q_h , \div u_h \bigr)& & &\!= \,& \, \bigl((q^I - q^P)^j,q_h \bigr). &
\end{array}
\end{equation}
For $2 \le j \le M$ find $c_h^j \in \mS_c^j$ such that, for all $w_h \in \mS_c^j$,
\begin{align} \label{scheme_2}
\bigl( \phi \, d_t c_h^j, w_h \bigr) + B(\overline{c_h}^j, w_h; \breve{u}_h^j) = \bigl(\overline{\hat{c}}^j {\overline{q}^I}^j, w_h \bigr)
\end{align}
and solve (\ref{scheme_1}) to obtain $(u_h^j, p_h^j) \in \mS_u^j \times \mS_p^j$.}

The algorithm only requires the solution of a {\em linear} system in each time step. The iterate $c_h^1$ can be computed with an implicit Euler method and fine time steps. The use of extrapolated values such as $\breve{u}_h^j$ is classical, e.g. see \cite[p.~218]{Thomee91}.

\section{Unconditional Well-posedness, Boundedness and Convergence} \label{sec_WPS}

Given $c_h^{j-1}$ and $c_h^{j-2}$, there exists a solution $c_h^j \in \mS_u^j$ of (\ref{scheme_2}) because the bilinear form $B$ is positive definite. For $t \in [t_{j-1}, t_j]$, let $\tilde{c}_h(t,\cdot) := \frac{t-t_{j-1}}{k_j} \, c_h^{j} + \frac{t_j - t}{k_j} \, c_h^{j-1}$. Then $\partial_t \tilde{c}_h(t, \cdot) = d_t c_h^j(\cdot)$. We interpret elements of $\mS_u$, $\mS_p$ and $\mS_c$ as time-dependent functions with stepwise constant values. Let
\begin{eqnarray*}
| c_h |_{\breve{u}_h}^2 &:=& \bigl( \bD_h(\breve{u}_h) \nablah c_h, \nablah c_h \bigr) + \bigl( \sigma^2 [c_h], [c_h] \bigr)_\mEi + \bigl( | \breve{u}_h \cdot n_\mEa| \; [c_h], [c_h] \bigr)_\mEi.
\end{eqnarray*}

\begin{theorem} \label{spatial_stab}
Let $\rho^\circ = \| \rho \|_\infty$. There exists a constant $C > 0$ such that
\begin{eqnarray} \label{bound_u}
\| \breve{u}_h^j \| + \|\div \breve{u}_h^j \| + \|\breve{p}_h^j\| \lesssim \bigl( \|\rho^\circ g \| + \|\breve{q}^I - \breve{q}^P \| \bigr)
\end{eqnarray}
holds for all $j =2,3 \ldots ,M$. Equally we have
\begin{eqnarray} \label{bound_c}
\|\phi^{1/2} c_h^j\|^2+ \int_{t_1}^{t_j} | \overline{c_h}|_{\breve{u}_h^j}^2 \dt \le \|\phi^{1/2} c_h^1\|^2 + \int_{t_1}^{t_j} \| \bigl({\overline{q}^I}^i\bigr)^{1/2} \, \overline{\hat{c}}^i \|^2 \dt
\end{eqnarray}
for all $j=2,3 \ldots, M$.
\end{theorem}
\begin{proof} The stability of $u^{j-1}$, $u^{j-2}$, $p^{j-1}$, $p^{j-2}$ follows from a classical inf-sup argument. This implies stability of $\breve{u}^j$ and $\breve{p}^j$. We choose $w_h = \overline{c_h}^i$ in (\ref{scheme_2}) to verify that
\begin{eqnarray*}
d_t \|\phi^{1/2} c_h^i\|^2 &+& | \overline{c_h}^i |_{\breve{u}_h^i}^2 + \| (\overline{q}^I + \overline{q}^P)^{1/2} \overline{c_h}^i \|^2 \le 2 \bigl( \phi \, d_t c_h^i, \overline{c_h}^i \bigr) + 2 B(\overline{c_h}^i, \overline{c_h}^i; \breve{u}_h^i) = 2 \bigl(\overline{\hat{c}}^i {\overline{q}^I}^i, \overline{c_h}^i \bigr).
\end{eqnarray*}
The Cauchy-Schwarz inequality, multiplication by $k_i$ and summation over $i$ give
\[\|\phi^{1/2} c_h^j\|^2+ \sum_{i=2}^j k_i  | \overline{c_h}^i |_{\breve{u}_h^i}^2  \le \|\phi^{1/2} c_h^1\|^2 + \sum_{i=2}^j k_i \| \bigl({\overline{q}^I}^i\bigr)^{1/2} \, \overline{\hat{c}}^i \|^2\]
for all $j=2,3,\ldots,M$.   \end{proof}

For simplicity the next theorem is stated assuming meshes are not adapted in time. For the extension to changing meshes consult \cite{BartelsJensenMueller08}. However, observe that that the discretisation with the implicit Euler method gives additional stability in $k_i \|\phi^{1/2} d_t c_h^i \|^2$, which allows to change meshes more rapidly.

\begin{theorem} \label{time_stab}
The time derivative $\partial_t \tilde{c}_h$ belongs to $L^2(t_1,T; \Wpm^*)$ and
\[
\| \partial_t \tilde{c}_h \|_{L^2(t_1,T; \Wpm^*)} = \| d_t c_h \|_{L^2(t_1,T; \Wpm^*)} \lesssim 1,
\]
independently of the mesh size and time step.
\end{theorem}
\begin{proof} Let $w_h \in \mS_c^j$. We recall from \cite{BartelsJensenMueller08}
\begin{align*}
B_d(c_h^j, w_h; \breve{u}_h^j) & \lesssim (1 + \| \breve{u}_h^j \|^\eoh) \, | c_h^j |_{\mT^j} \, (\| \nablah w_h \|_{L^4(\Omega)} + \| w_h \|_{L^4(\Omega)} + \| \sigma [w_h] \|_\mEi),\\
B_{cq}(c_h^j, w_h; \breve{u}_h^j) & \lesssim (1 + \|\breve{u}_h^j\|^\eoh ) \, | c_h^j |_{\mT^j} \, (\| \nablah w_h \| + \| w_h \|_{L^4(\Omega)} + \| \sigma [w_h] \|_\mEi),\\
\| \sigma [w_h] \|_\mEi^2 & \lesssim (1 + \| \breve{u}_h^j \|) \, \tilde{h}^{\eoh} \| w \|_{H^2(\O)}^2.
\end{align*}
With $L^2$-orthogonality and
\begin{eqnarray*} \label{GO} \nonumber
\int_{t_1}^T \bigl( \phi \, d_t c_h^j, w \bigr) \dt &=& \int_{t_1}^T  \bigl(\overline{\hat{c}}^j {\overline{q}^I}^j, w_h \bigr) - B(\overline{c_h}^j, w_h; \breve{u}_h^j) \dt\\  \nonumber
& \lesssim & \int_0^T (1 + \|\breve{u}_h^j\| ) (1 + \|\breve{u}_h^j\|_{\Hdiv}^\eoh ) \, | c_h^j |_{u_h^j} \, \| w \|_{\Wpm} \dt\\
& \lesssim & \| w \|_{L^2(0,T; \Wpm)}
\end{eqnarray*}
one completes the proof.   \end{proof}

\begin{theorem}\label{dg_limits_1}
Let $(u_i, p_i, c_i)_{i \in \dN}$ be a sequence of numerical solutions with $(\tilde{h}_i, \tilde{k}_i) \to 0$ as $i \to \infty$. Then there exists $c \in L^2(0, T; H^1(\O)) \cap H^1(0,T;\Wpm^{*})$ such that, after passing to a subsequence, $c_i \to c $ in $L^{2}(\OT)$, $\partial_{t} \tilde{c}_i \rightharpoonup \partial_{t} c$ in $L^{2}(0,T;\Wpm^{*})$ and $\nabla c_i \rightharpoonup \nabla c$ in $L^2(0,T;H^{-1}(\O))$. If $c_i^0, c_i^1 \to c_0$ in $\Wpm^{*}$ then $c$ satisfies {\rm (W3)}.
\end{theorem}

The proof is, up to the treatment of the initial conditions, exactly as in \cite{BartelsJensenMueller08}. It is based on the Aubin-Lions theorem and the embedding
\[
\mS^{\pds}(\mT_i) \hookrightarrow [{\rm BV}(\O) \cap L^4(\O), L^4(\O)]_{1/2} \hookrightarrow L^2(\Omega),
\]
where $[\cdot,\cdot]_{\theta}$ denotes the complex method of interpolation.

\begin{samepage}
\begin{theorem}\label{contu}
Let $(u_i, p_i, c_i)_{i \in \dN}$ be numerical solutions with $(\tilde{h}_i, \tilde{k}_i) \to 0$ and $c_i \to c$ in $L^{2}(\OT)$ as $i \to \infty$. There exists $u \in L^\infty(0, T; \HNdiv)$ and $p\in L^\infty(0, T; L^2_0(\O))$ such that, after passing to a subsequence, $u_i \to u$ in $\HNdiv$ and $p_i \to p$ in $L^2_0(\O)$ as $(\tilde{h}_i, \tilde{k}_i) \to 0$. Furthermore, $(u, p, c)$ satisfies {\rm (W1)}.
\end{theorem}

\begin{proof} Use Strang's lemma, for details see \cite{BartelsJensenMueller08}.   \end{proof}
\end{samepage}

We interpret $\breve{u}_i$ as piecewise constant function in time, attaining in $(t_{j-1}, t_j]$ the value ${\textstyle \frac{3}{2}} u(t^{j-1}) - {\textstyle \frac{1}{2}} u(t^{j-2})$.

\begin{theorem} \label{weaksol}
Let $(u_i, p_i, c_i)_{i \in \dN}$ be a sequence of numerical solutions with $(\tilde{h}_i, \tilde{k}_i) \to 0$ as $i \to \infty$ and let $u \in L^\infty(0,T;H_N(\div;\O))$ and $c \in L^2(0, T; H^1(\O)) \cap H^1(0,T;\Wpm^{*})$ be a limit of $(u_i, c_i)_i$ in the sense of Theorems \ref{dg_limits_1} and \ref{contu}. Then $(u,c)$ satisfies {\rm (W2)}.
\end{theorem}

\begin{proof} Let $v \in \sD(0,T; \sC^\infty(\O))$ and $v_i(t) \in \mS_c^j$ an approximation to $v(t)$ in $(t_{j-1}, t_j]$. Using the strong convergence of $(\nablah v_i)_i$ in $L^\infty(\OT)^d$ and the weak convergence of the lifted gradient of $c_i$ in $L^2(\OT)^d$, we find
\begin{eqnarray} \nonumber \label{approxgraduse}
\int_{t_1}^T \!\! \bigl( \nabla c, \bD(u) \nabla v \bigr) \dt = \lim_{i \to \infty} \int_{t_1}^T \bigl(\nablah c_i, \bD_h(\breve{u}_i) \nablah v_i\bigr) - \bigl( [c_i], \{ \bD_h(\breve{u}_i) \nablah v_i \}\bigr)_{\mE_\O} \dt.
\end{eqnarray}
As in \cite{BartelsJensenMueller08} it follows that $B_d(c_i, v_i; \breve{u}_i)$ coincides in the limit with $\bigl( \nabla c, \bD(u) \nabla v \bigr)$. One can also conclude by adapting \cite{BartelsJensenMueller08} that
\[
\int_{t_1}^T \bigl(u \cdot \nabla c, v \bigr) + \bigl(q^I c, v \bigr) \dt = \lim_{i \to \infty} \int_{t_1}^T B_{cq}(c_i, v_i; \breve{u}_i) \dt.
\]
One arrives at
\begin{eqnarray*}
&& \hphantom{\lim_{i \to \infty}} \int_{t_1}^T - \bigl( \phi \, c, \partial_t v \bigr) + \bigl(\bD(u) \nabla c, \nabla v \bigr) + \bigl(u \cdot \nabla c, v \bigr) + \bigl(q^I c, v \bigr) - \bigl(\hat{c} q^I, v \bigr) \dt\\
&=& \lim_{i \to \infty} \int_{t_1}^T \bigl( \phi \, d_t c_h^j, w_h \bigr) + B(\overline{c_h}^j, w_h; \breve{u}_h^j) - \bigl(\overline{\hat{c}}^j {\overline{q}^I}^j, w_h \bigr) \dt = 0.
\end{eqnarray*}
Hence (W2) is satisfied for $v \in \sD(0,T; \sC^\infty(\O))$. The extension to $\sD(0,T; \Wpm)$ follows from boundedness and density of smooth functions.
\end{proof}

\begin{figure}
\begin{center}
\includegraphics[height=26ex]{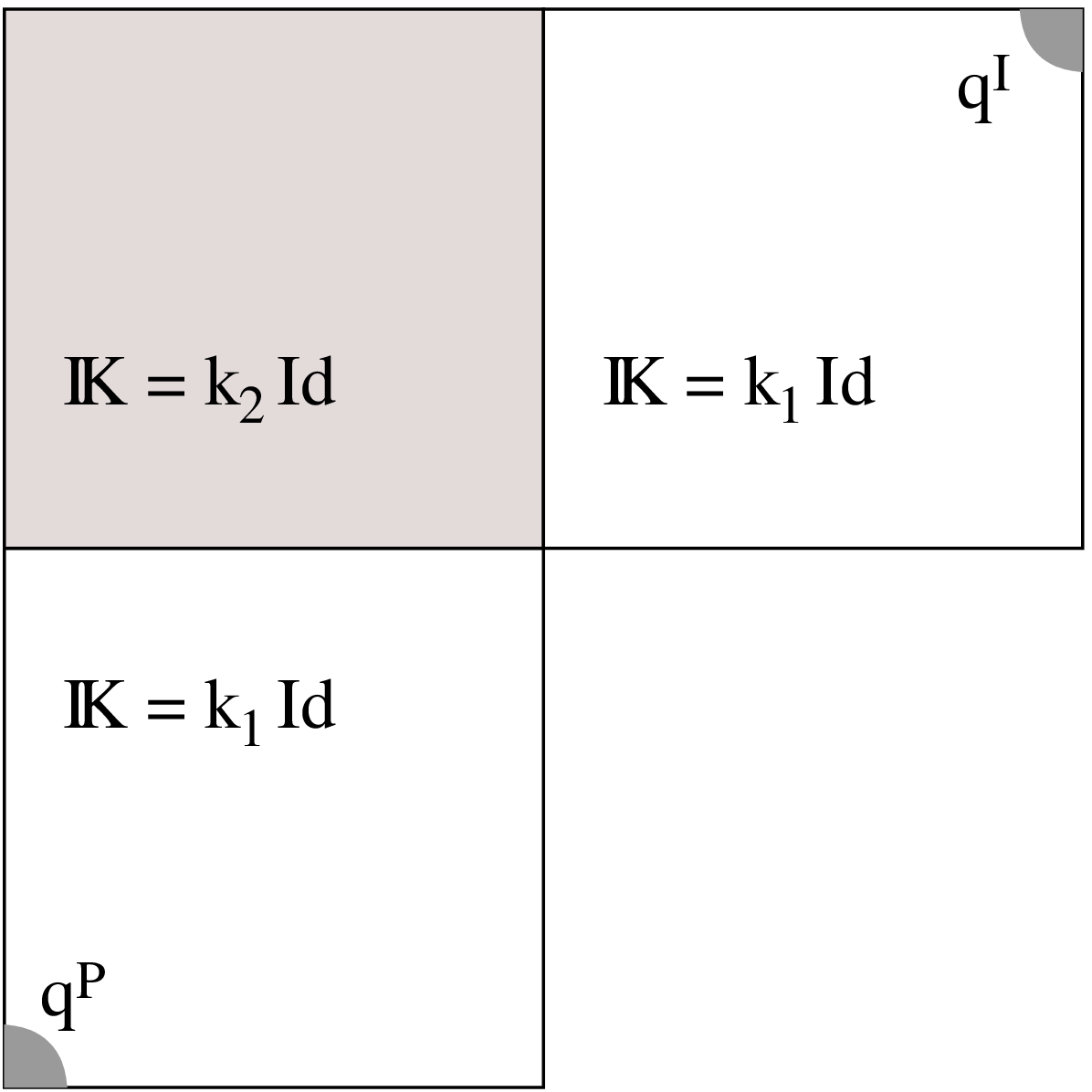} \hspace{1cm}
\includegraphics[height=30ex,trim=0 2 0 90, clip]{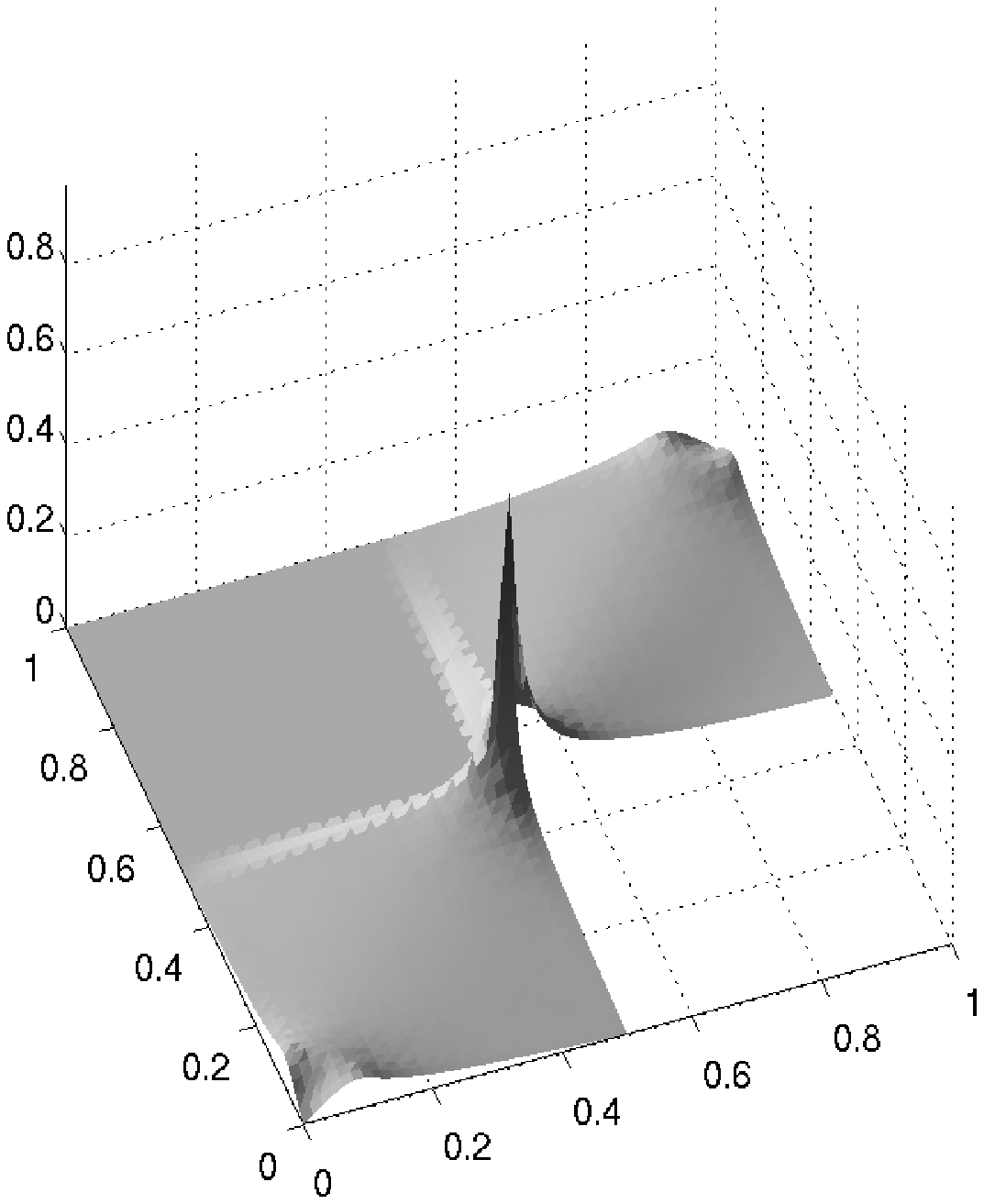} \hspace{1cm}
\caption{\emph{Example 1: } Left: computational domain; right: absolute value $|u_h|$ of the Darcy velocity at $t=1.0$ before any interaction between the concentration front and the corner singularity.
\label{fig:lshape-domain}}
\end{center}
\end{figure}

\section{Numerical Experiments}

\begin{figure}
\begin{center}
\includegraphics[width=0.30\textwidth]{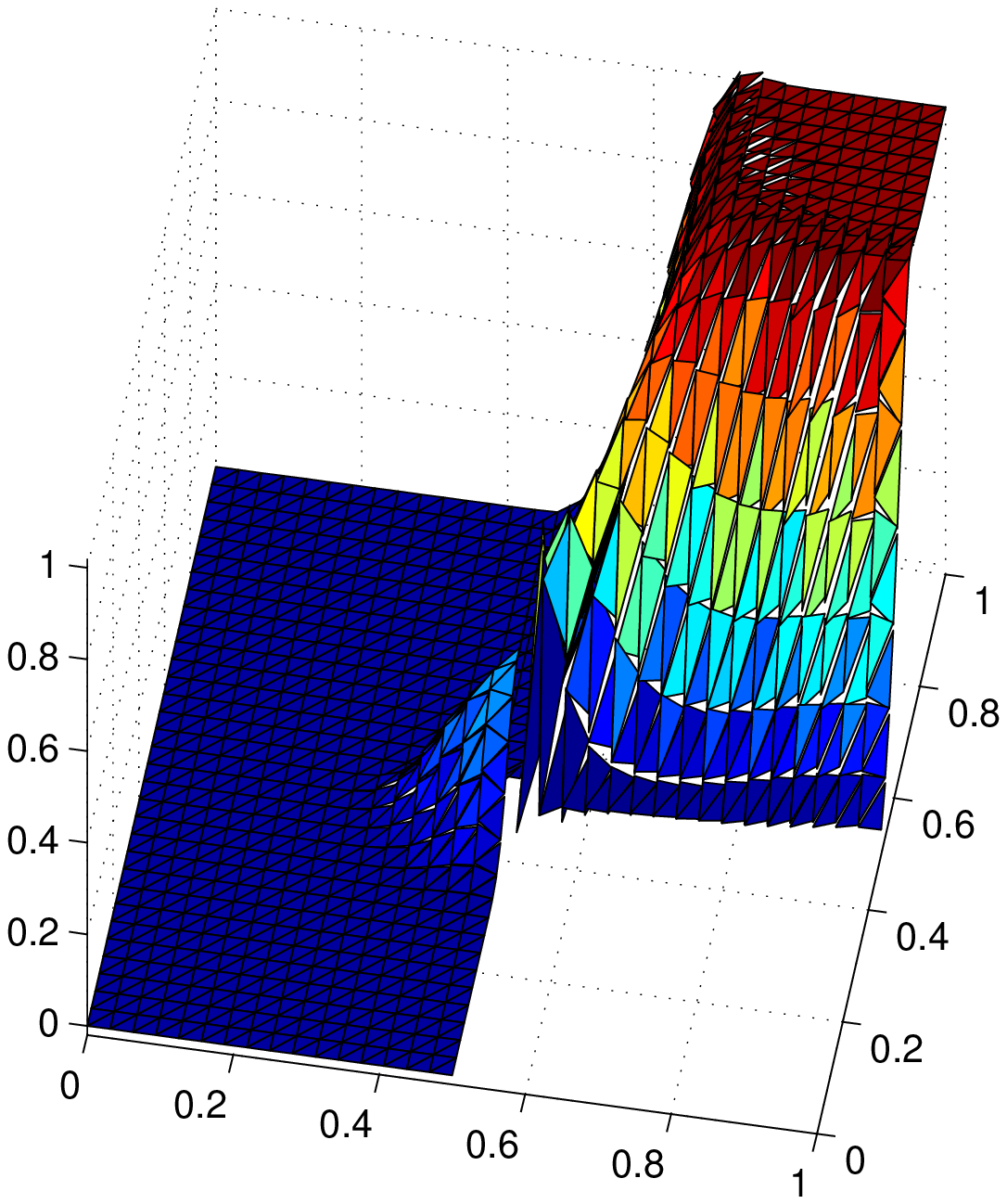}\qquad
\includegraphics[width=0.30\textwidth]{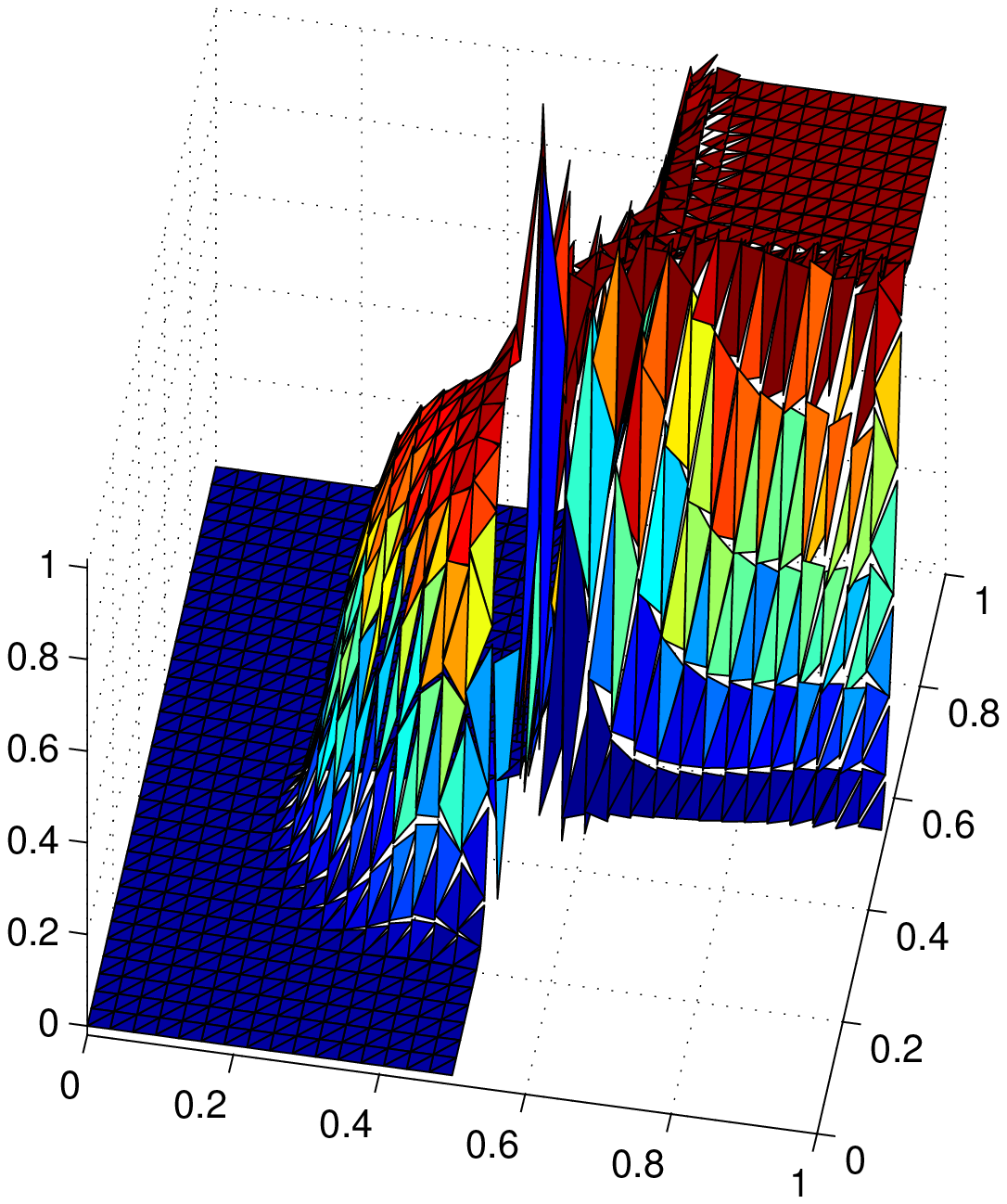}
\caption{Snapshots of $c_h$ at $t=1.5$ and $2.0$, computed with the Crank-Nicolson scheme.
\label{fig:lshape}}
\end{center}
\end{figure}

The numerical experiments are carried out in two space dimensions with the lowest-order method on a mesh which consists of shape-regular triangles without hanging nodes and which is not changed over time.
The diffusion--dispersion tensor takes the form
\begin{equation} \label{D_def}
  \bD(u,x) = \phi(x) \left( d_{m} {\rm Id} + |u| \, d_{\ell} \, E(u) + |u| \, d_{t} \left( {\rm Id} - \, E(u) \right) \right).
\end{equation}

\begin{figure}[ht]
\begin{center}
\includegraphics[width=0.3\textwidth]{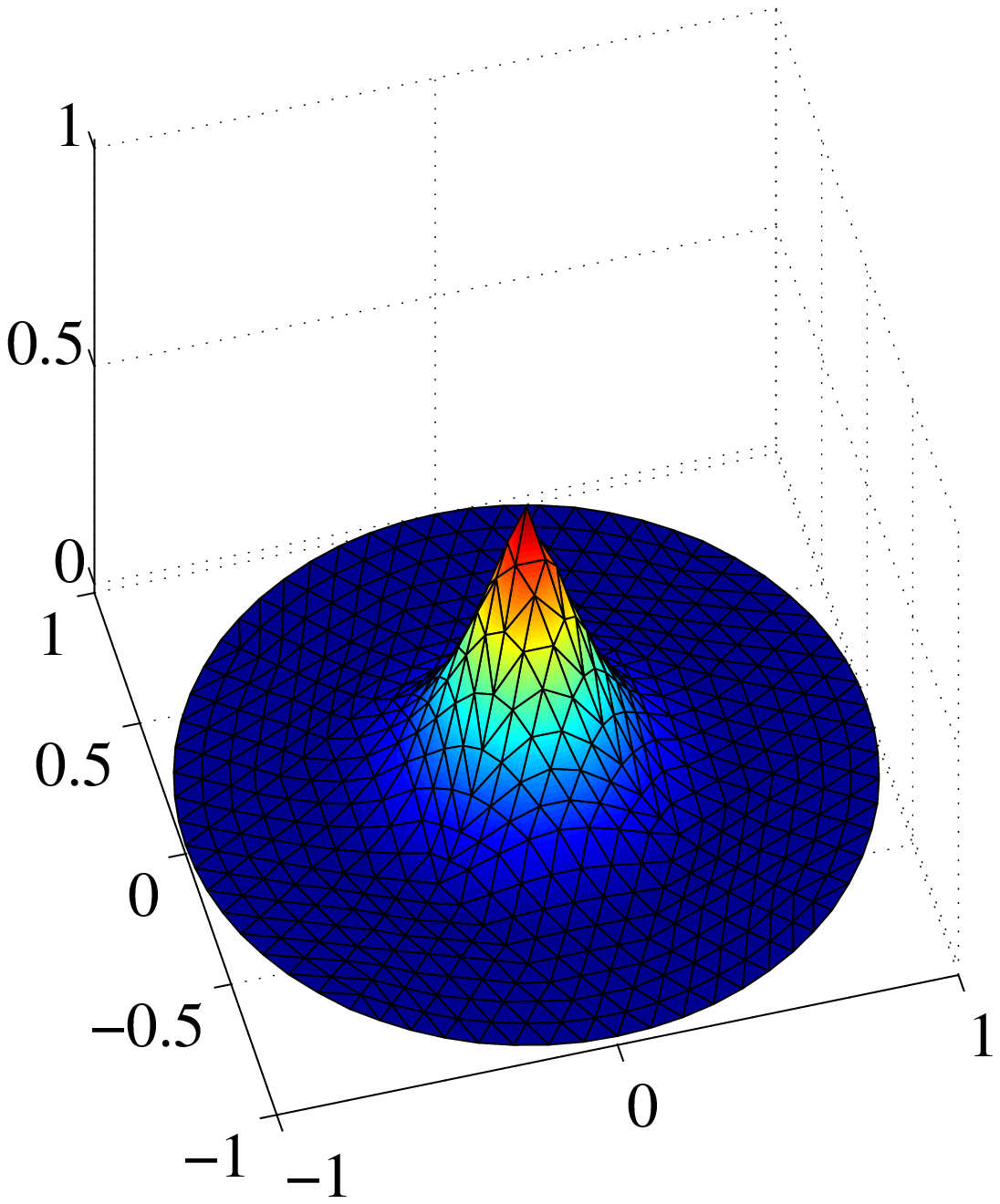}
\includegraphics[width=0.3\textwidth]{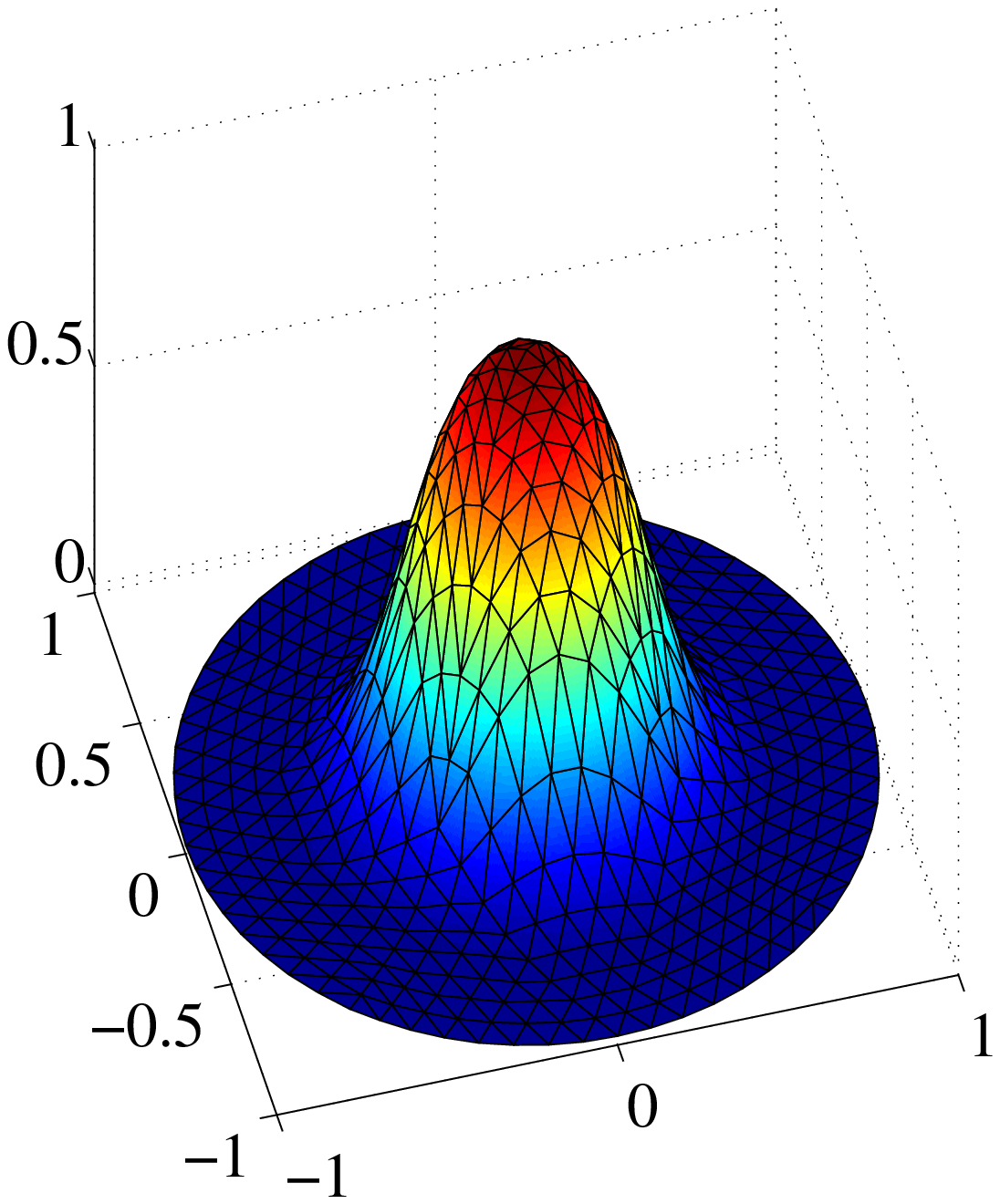}
\includegraphics[width=0.3\textwidth]{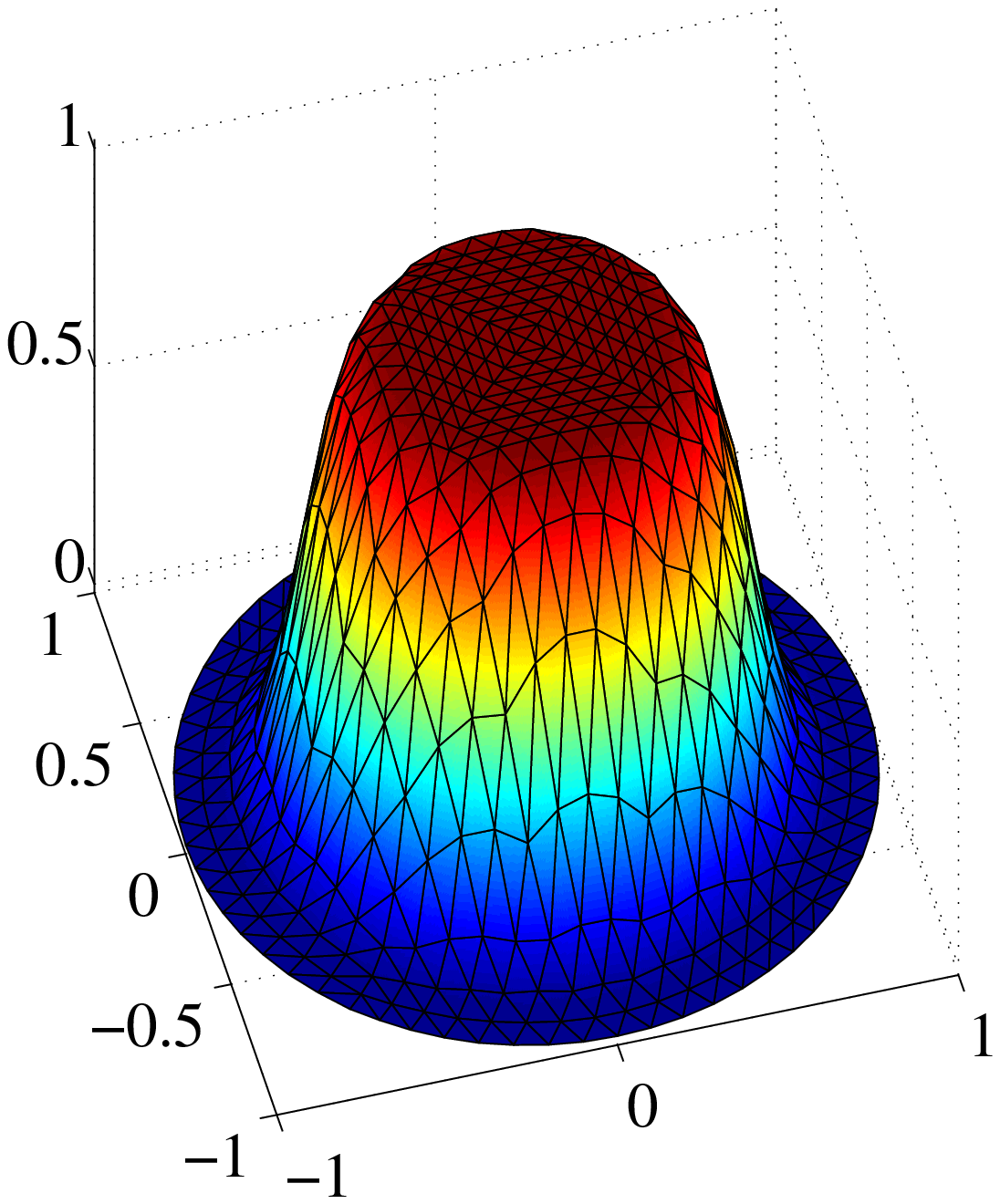}
\caption{\emph{Example 2: } Snapshots of the concentration $c_{\mathrm{ref}}$ at $t=0.25, 1.0$ and $3.0$.
\label{fig:ref_sol}}
\end{center}
\end{figure}

\begin{numexample}[Singular Velocities] \rm
To examine the effect of a singular velocity field caused by a discontinuous permeability distribution and a re-entrant corner we employ the L-shaped domain $\Omega$ and $\K$ with $k_1=0.1$ and $k_2=10^{-6}$ as depicted in Figure~\ref{fig:lshape-domain}. The injection and production wells are located at $(1,1)$ and $(0,0)$, respectively. The porous medium is almost impenetrable in the upper left quarter, forcing a high fluid velocity at the reentrant corner where the nearly impenetrable barrier is thinnest. This leads to a singularity $|u|\sim r^{-\alpha}$, where $r$ is the distance to the reentrant corner and $\alpha\approx 1$, cf.~\cite{BartelsJensenMueller08}. Figure \ref{fig:lshape} shows the concentration when the front passes the corner and at a later time. The solution $c_h$ contains steep fronts but shows only the localised oscillations that are characteristic for dG methods.
\end{numexample}

\begin{figure}[ht]
\begin{center}
\includegraphics[width=0.5\textwidth]{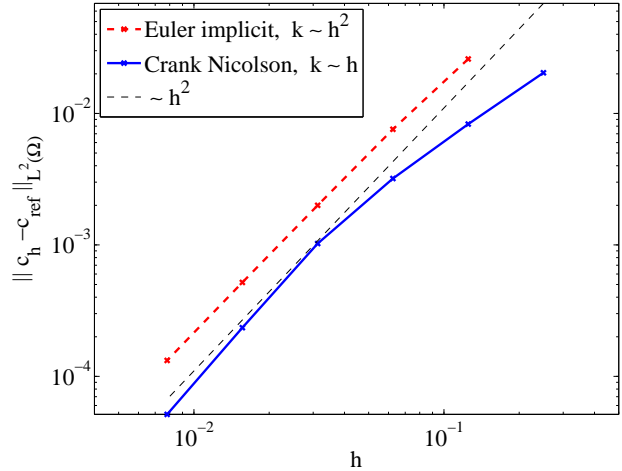}
\caption{Error $\|c_h -c_{\mathrm{ref}}\|_{L^2(\Omega)}$ of the implicit Euler method the Crank-Nicolson method at time $t=1$.
\label{fig:rates}}
\end{center}
\end{figure}

\begin{numexample}[Convergence rates] \rm
Convergence rates are determined by comparing the numerical solution $c_h$ to a reference solution $c_{\mathrm{ref}}$ that is computed with high accuracy on a one dimensional grid. More precisely, we set $\phi=1$, $\hat{c}=1$, $\K=1$ and $g=0$ and choose $\Omega$ to be the ball $B(0,1)\subset\dR^2$. Using polar coordinates $(r,\varphi)$, we choose $q^I = 4 \left( 1-r \right)^{6}$ and $q^P = \frac{4}{7} r^{6}$. Then the Darcy velocity only changes in the radial direction and is determined by an ODE, which has the nonnegative exact solution $u(r) = \frac{r}{7} \left( 3\,r^6 -24\,r^5 +70\,r^4 -112\,r^3 +105\,r^2 -56\,r +14 \right)$. Consequently, the concentration equation reduces to a linear parabolic equation in one space dimension. Figure \ref{fig:ref_sol} shows snapshots of the solution $c_{\mathrm{ref}}$ with $d_{m}=1.0\times 10^{-5}$, $d_{\ell}=4.0\times 10^{-4}$ and Figure \ref{fig:rates} shows that $L^2$ error of implicit Euler method is of order $O(h^2 + k)$ whereas the Crank-Nicolson reaches the order $O(h^2 + k^2)$.

\end{numexample}

\end{document}